\documentclass[final, nomarks]{dmtcs-episciences}
\usepackage{amssymb,amsthm,amsmath,amsfonts,latexsym,tikz,hyperref,enumitem,tikz-cd}

\newtheorem{thm}{Theorem}[section]
\newtheorem{prop}[thm]{Proposition}
\newtheorem{cor}[thm]{Corollary}
\newtheorem{lem}[thm]{Lemma}
\newtheorem{conj}[thm]{Conjecture}
\newtheorem{exa}[thm]{Example}
\newtheorem{defn}[thm]{Definition}

\newcommand{\pro}{\partial}
\newcommand{\da}{\downarrow}

\newcommand{\bdu}{\hs{3pt}\hat{\cup}\hs{3pt}}

\newcommand{\ben}{\begin{enumerate}}
\newcommand{\een}{\end{enumerate}}
\newcommand{\ble}{\begin{lem}}
\newcommand{\ele}{\end{lem}}
\newcommand{\bth}{\begin{thm}}
\renewcommand{\eth}{\end{thm}}
\newcommand{\bpr}{\begin{prop}}
\newcommand{\epr}{\end{prop}}
\newcommand{\bco}{\begin{cor}}
\newcommand{\eco}{\end{cor}}
\newcommand{\bcon}{\begin{conj}}
\newcommand{\econ}{\end{conj}}
\newcommand{\bde}{\begin{defn}}
\newcommand{\ede}{\end{defn}}
\newcommand{\bex}{\begin{exa}}
\newcommand{\eex}{\end{exa}}
\newcommand{\barr}{\begin{array}}
\newcommand{\earr}{\end{array}}
\newcommand{\btab}{\begin{tabular}}
\newcommand{\etab}{\end{tabular}}
\newcommand{\beq}{\begin{equation}}
\newcommand{\eeq}{\end{equation}}
\newcommand{\bea}{\begin{eqnarray*}}
\newcommand{\eea}{\end{eqnarray*}}
\newcommand{\bal}{\begin{align*}}
\newcommand{\bce}{\begin{center}}
\newcommand{\ece}{\end{center}}
\newcommand{\bpi}{\begin{picture}}
\newcommand{\epi}{\end{picture}}
\newcommand{\bpp}{\begin{picture}}
\newcommand{\epp}{\end{picture}}
\newcommand{\bfi}{\begin{figure} \begin{center}}
\newcommand{\efi}{\end{center} \end{figure}}
\newcommand{\bprf}{\begin{proof}}
\newcommand{\eprf}{\end{proof}\medskip}

\newcommand{\bsl}{\begin{slide}{}}
\newcommand{\esl}{\end{slide}}
\newcommand{\bfr}{\begin{frame}}
\newcommand{\efr}{\end{frame}}

\newcommand{\hqed}{\hfill \qed}

\newcommand{\eqqed}[1]{$\rule{1ex}{0ex}\hfill{\dil#1}\hfill\qed$}

\newcommand{\hs}[1]{\hspace{#1}}
\newcommand{\hso}[1]{\hspace{-1pt}}
\newcommand{\vs}[1]{\vspace{#1}}
\newcommand{\qmq}[1]{\quad\mbox{#1}\quad}

\newcommand{\sbe}{\subseteq}

\newcommand{\zh}{\hat{0}}
\newcommand{\oh}{\hat{1}}

\newcommand{\lt}{\lhd}
\newcommand{\gt}{\rhd}

\newcommand{\case}[4]{\left\{\barr{ll}#1&\mbox{#2}\\#3&\mbox{#4}\earr\right.}

\def\<{\langle}
\def\>{\rangle}

\newcommand{\spn}[1]{\langle{#1}\rangle}

\newcommand{\ra}{\rightarrow}

\newcommand{\om}{\omega}

\newcommand{\bbN}{{\mathbb N}}
\newcommand{\bbP}{{\mathbb P}}

\newcommand{\cJ}{{\cal J}}

\newcommand{\cL}{{\cal L}}

\newcommand{\cM}{{\cal M}}
\newcommand{\cN}{{\cal N}}

\newcommand{\cO}{{\cal O}}
\newcommand{\cP}{{\cal P}}

\DeclareMathOperator{\lcm}{lcm}

\DeclareMathOperator{\rk}{rk}

\DeclareMathOperator{\st}{st}

\newcommand{\dil}{\displaystyle}

\author{Jamie Kimble
\and
Bruce E. Sagan
\and
Avery St.\ Dizier
}

\title{$K$-promotion on $m$-packed labelings of posets
}

\affiliation{Department of Mathematics, Michigan State University, East Lansing, MI, USA}


\begin{document}

\publicationdata{vol. 28:2}{2026}{16}{10.46298/dmtcs.16351}{2025-08-14; 2025-08-14; 2026-01-07; 2026-02-18; 2026-02-23}{2026-02-24}

\maketitle

\begin{abstract}
\vspace{10pt}
Sch\"utzenberger's promotion operator $\pro$ is a fundamental map in dynamical algebraic combinatorics.  At first, its action was mainly considered on standard Young tableaux.
But $\pro$ was subsequently shown to have interesting properties when applied to natural labelings of other posets.  Pechenik defined a $K$-theoretic version of promotion,  $\pro_K$, on $m$-packed labelings of tableaux.  The operator  $\pro_K$ was then extended  to increasing labelings of other posets. The purpose of the current work is to show that the original action of $\pro_K$ on $m$-packed labelings yields interesting results when applied to  partially ordered sets in general, and to rooted trees in particular.   We show that under certain conditions, the sizes of the orbits and order of $\pro_K$ exhibit nice divisibility properties.  We also completely determine, for certain values of $m$, the orbit sizes for the action on various types of  rooted trees such as extended stars, combs, zippers, and a type of three-leaved tree.
\end{abstract}

%
%
%
%
%



\section{Introduction}

We let $\bbN$ and $\bbP$ be the nonnegative and positive integers, respectively.
For  $m,n\in\bbN$, we let
$$
[m,n]=\{m, m+1,\ldots,n\} \qmq{and} [n]=[1,n].
$$
All of our sets $S$  will be finite and we will denote by $\#S$ or $|S|$ the cardinality of $S$.

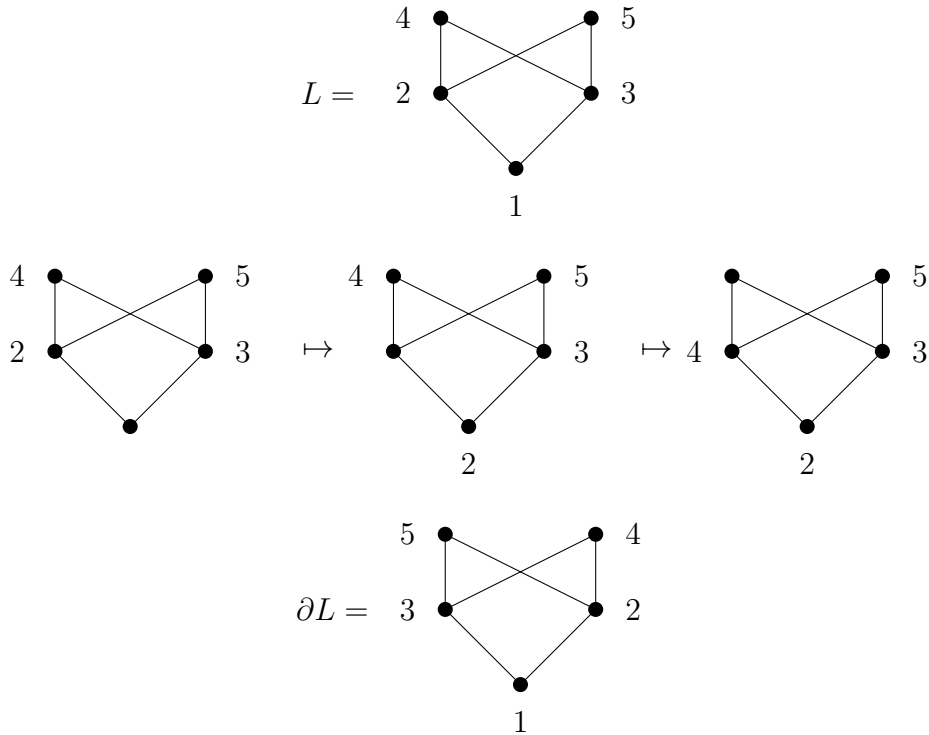
\begin{figure}
    \centering
\begin{tikzpicture}
\draw(-1.5,1) node{$L=$};
\fill(1,0) circle(.1);
\fill(0,1) circle(.1);
\fill(2,1) circle(.1);
\fill(0,2) circle(.1);
\fill(2,2) circle(.1);
\draw (1,0)--(0,1)--(0,2)--(2,1)--(2,2)--(0,1) (1,0)--(2,1);
\draw(1,-.5) node{$1$};
\draw(-.5,1) node{$2$};
\draw(2.5,1) node{$3$};
\draw(-.5,2) node{$4$};
\draw(2.5,2) node{$5$};
\end{tikzpicture}

\vs{10pt}

\begin{tikzpicture}
\fill(1,0) circle(.1);
\fill(0,1) circle(.1);
\fill(2,1) circle(.1);
\fill(0,2) circle(.1);
\fill(2,2) circle(.1);
\draw (1,0)--(0,1)--(0,2)--(2,1)--(2,2)--(0,1) (1,0)--(2,1);
\draw(-.5,1) node{$2$};
\draw(2.5,1) node{$3$};
\draw(-.5,2) node{$4$};
\draw(2.5,2) node{$5$};
\draw(3.5,1) node{$\mapsto$};
\begin{scope}[shift={(4.5,0)}]
\fill(1,0) circle(.1);
\fill(0,1) circle(.1);
\fill(2,1) circle(.1);
\fill(0,2) circle(.1);
\fill(2,2) circle(.1);
\draw (1,0)--(0,1)--(0,2)--(2,1)--(2,2)--(0,1) (1,0)--(2,1);
\draw(1,-.5) node{$2$};
\draw(2.5,1) node{$3$};
\draw(-.5,2) node{$4$};
\draw(2.5,2) node{$5$};
\draw(3.5,1) node{$\mapsto$};
\end{scope}
\begin{scope}[shift={(9,0)}]
\fill(1,0) circle(.1);
\fill(0,1) circle(.1);
\fill(2,1) circle(.1);
\fill(0,2) circle(.1);
\fill(2,2) circle(.1);
\draw (1,0)--(0,1)--(0,2)--(2,1)--(2,2)--(0,1) (1,0)--(2,1);
\draw(1,-.5) node{$2$};
\draw(-.5,1) node{$4$};
\draw(2.5,1) node{$3$};
\draw(2.5,2) node{$5$};
\end{scope}
\end{tikzpicture}

\vs{10pt}

\begin{tikzpicture}
\draw(-1.5,1) node{$\pro L=$};
\fill(1,0) circle(.1);
\fill(0,1) circle(.1);
\fill(2,1) circle(.1);
\fill(0,2) circle(.1);
\fill(2,2) circle(.1);
\draw (1,0)--(0,1)--(0,2)--(2,1)--(2,2)--(0,1) (1,0)--(2,1);
\draw(1,-.5) node{$1$};
\draw(-.5,1) node{$3$};
\draw(2.5,1) node{$2$};
\draw(-.5,2) node{$5$};
\draw(2.5,2) node{$4$};
\end{tikzpicture}
    
    \caption{The promotion operator $\pro$}
    \label{ProFig}
\end{figure}

Let us first recall Sch\"utzenberger's~\cite{sch:pme} promotion operator, $\pro$, on naturally labeled posets.  It was inspired by the work of Robinson~\cite{rob:rsg} and Knuth~\cite{knu:pmg} on Young tableaux.
Let $(P,\le)$ be a partially ordered set (poset) where we write $\le_P$ for $\le$ if we need to be more precise.  We will use the notation $x\lt y$ if $x$ is covered by $y$ in $P$.
For any terms from the theory of posets not defined here, see the texts of Sagan~\cite{sag:aoc} or Stanley~\cite{sta:ec1}.  If $\#P=n$ then a {\em natural labeling} of $P$ is a bijection $L:P\ra[n]$ such that
\beq
\label{NatL}
x<_P y \qmq{implies} L(x)<L(y).
\eeq
We let $\cN(P)$ be the set of natural labelings of $P$.
The top and bottom posets in Figure~\ref{ProFig} are naturally labeled.  Now, if $z\in P$ then one can have a {\em gapped natural labeling} of $P$ which is a bijection $L:P-\{z\}\ra[2,n]$ which satisfies~\eqref{NatL} for all elements in the domain of $L$.  The element $z$ is considered {\em unlabeled}.  All the labelings in the middle row of Figure~\ref{ProFig} are of this form.

Given a poset $P$ with $\#P=n$, the {\em promotion operator} is $\pro:\cN(P)\ra\cN(P)$ defined as follows.  Suppose we are given $L\in\cN(P)$.  Remove the $1$ from $L$ and let $z\in P$ be the element which is now unlabeled.  Of all the elements which cover $z$, let $w$ be the one of minimum label.  Form a new gapped natural labeling of $P$ by letting $L(z)=L(w)$ and removing the label from $w$.  Iterate this process until the unlabeled element is maximal in $P$.  Now decrease the labels of all labeled elements by one and label the unlabeled maximal element as $n$ to form $\pro L$.  An example is worked out in Figure~\ref{ProFig}.
This is one of the fundamental operations in the area which has come to be known as dynamical algebraic combinatorics.  See the survey articles of Roby~\cite{rob:dac} or Striker~\cite{str:dac} for more information.

An {\em increasing labeling} of a poset $P$ is a function $f:P\ra\bbP$ satisfying restriction~\eqref{NatL}.  Furthermore, we say that $f$ is {\em $m$-packed} if it is increasing and the image of $f$ is $[m]$ for some $m\in\bbP$.
Note that this forces $m\le |P|$.  And if $m=|P|$ then $L$ is a natural labeling.  
We let
$$
\cL_m(P) = \{L : \text{$L$ is an $m$-packed labeling of $P$}\}.
$$
  The top and bottom labelings in Figure~\ref{ProKFig} are $4$-packed.  Now let $A$ be an antichain of $P$. A {\em gapped $m$-packed labeling} of $P$ is a surjection $L:P- A\ra[2,m]$ satisfying~\eqref{NatL} on $P- A$.  The labelings in the middle row of Figure~\ref{ProKFig} are of this type.

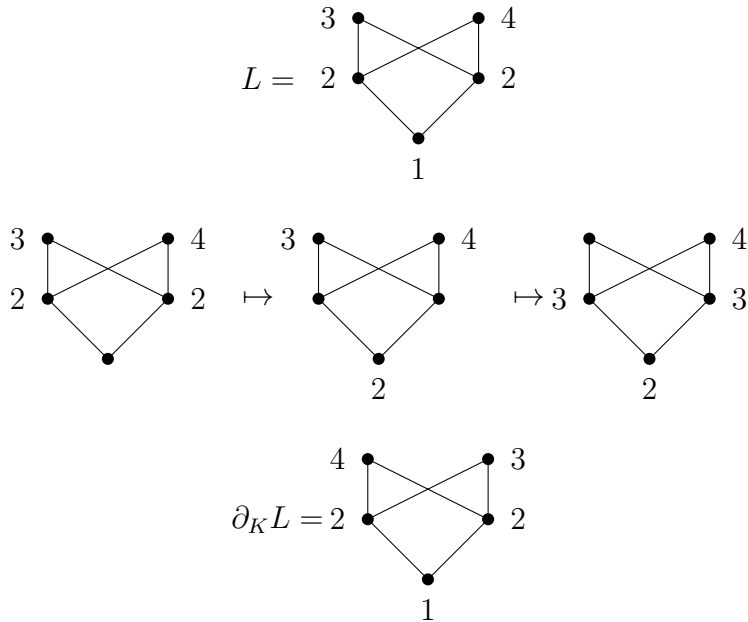
\begin{figure}
    \centering
\begin{tikzpicture}[scale=.8]
\draw(-1.5,1) node{$L=$};
\fill(1,0) circle(.1);
\fill(0,1) circle(.1);
\fill(2,1) circle(.1);
\fill(0,2) circle(.1);
\fill(2,2) circle(.1);
\draw (1,0)--(0,1)--(0,2)--(2,1)--(2,2)--(0,1) (1,0)--(2,1);
\draw(1,-.5) node{$1$};
\draw(-.5,1) node{$2$};
\draw(2.5,1) node{$2$};
\draw(-.5,2) node{$3$};
\draw(2.5,2) node{$4$};
\end{tikzpicture}

\vs{10pt}

\begin{tikzpicture}[scale=.8]
\fill(1,0) circle(.1);
\fill(0,1) circle(.1);
\fill(2,1) circle(.1);
\fill(0,2) circle(.1);
\fill(2,2) circle(.1);
\draw (1,0)--(0,1)--(0,2)--(2,1)--(2,2)--(0,1) (1,0)--(2,1);
\draw(-.5,1) node{$2$};
\draw(2.5,1) node{$2$};
\draw(-.5,2) node{$3$};
\draw(2.5,2) node{$4$};
\draw(3.5,1) node{$\mapsto$};
\begin{scope}[shift={(4.5,0)}]
\fill(1,0) circle(.1);
\fill(0,1) circle(.1);
\fill(2,1) circle(.1);
\fill(0,2) circle(.1);
\fill(2,2) circle(.1);
\draw (1,0)--(0,1)--(0,2)--(2,1)--(2,2)--(0,1) (1,0)--(2,1);
\draw(1,-.5) node{$2$};
\draw(-.5,2) node{$3$};
\draw(2.5,2) node{$4$};
\draw(3.5,1) node{$\mapsto$};
\end{scope}
\begin{scope}[shift={(9,0)}]
\fill(1,0) circle(.1);
\fill(0,1) circle(.1);
\fill(2,1) circle(.1);
\fill(0,2) circle(.1);
\fill(2,2) circle(.1);
\draw (1,0)--(0,1)--(0,2)--(2,1)--(2,2)--(0,1) (1,0)--(2,1);
\draw(1,-.5) node{$2$};
\draw(-.5,1) node{$3$};
\draw(2.5,1) node{$3$};
\draw(2.5,2) node{$4$};
\end{scope}
\end{tikzpicture}

\vs{10pt}

\begin{tikzpicture}[scale=.8]
\draw(-1.5,1) node{$\pro_K L=$};
\fill(1,0) circle(.1);
\fill(0,1) circle(.1);
\fill(2,1) circle(.1);
\fill(0,2) circle(.1);
\fill(2,2) circle(.1);
\draw (1,0)--(0,1)--(0,2)--(2,1)--(2,2)--(0,1) (1,0)--(2,1);
\draw(1,-.5) node{$1$};
\draw(-.5,1) node{$2$};
\draw(2.5,1) node{$2$};
\draw(-.5,2) node{$4$};
\draw(2.5,2) node{$3$};
\end{tikzpicture}
    
    \caption{The $K$-promotion operator $\pro_K$}
    \label{ProKFig}
\end{figure}

We now have everything in place to define, given $P$ and $m$, the 
{\em $K$-promotion function} $\pro_K:\cL_m(P)\ra\cL_m(P)$.  Consider an $m$-packed labeling $L:P\ra[m]$.
Since $L$ is increasing, the elements labeled $1$ form an antichain, $A$.  Remove the labels from $A$.  Now for all covers $x\lt y$ where  $x$ is unlabeled and $L(y)=2$, we label $x$ with $2$, remove the label from $y$, and leave all elements not in such a pair the same.  This is how we get from the first labeling to the second in row two of Figure~\ref{ProKFig}.  We now iterate the process, where the next step will consider covers $v\lt w$ with $v$ unlabeled and $L(w)=3$, and so forth.  
Termination occurs when the unlabeled elements are all maximal.
This completes the second row in Figure~\ref{ProKFig}. It is easy to prove by induction that at every stage the set of unlabeled elements forms an antichain.  Finally, we decrement the label of every labeled element by one and label the unlabeled maximal elements with $m$.  Again, a simple induction shows that the new labeling is $m$-packed and so $\pro_K(L)\in\cL_m(P)$, making $\pro_K$ a well-defined map.

Pechenik~\cite{pec:csi} first defined $\pro_K$, on certain 
$m$-packed Young tableaux. He was inspired by work of  Thomas and Yong~\cite{TY:jtt}.   Note that he called such tableaux ``increasing"  even though they also satisfied the condition on the image given above.
Dilks, Striker, and Vorland~\cite[Definition 3.3]{DSV:ril} generalized the definition of $\pro_K$ to all increasingly labeled posets. 
Such actions have been considered  in a number of other papers~\cite{BSV:ppb,BSV:ppp,BSV:ppq,DPS:rop,GPSS:ccs}.
The purpose of the present work is to show how keeping the surjective assumption  yields interesting results about $\pro_K$ both for posets in general and rooted tree posets in particular.  
In fact, it has been shown that to study the order of promotion of non-packed labelings, it suffices to look at packed labelings; see the papers of Pechenik and 
Mandel~\cite[Theorem 6.1]{MP:opp} for minuscule posets or Banaian, Barnard, Chepuri, and Striker~\cite{BBCS:ops} for posets in general.

 Note that orbits of $K$-promotion on $m$-packed labelings of any poset are a subset of those for its increasing labelings.  So, some of the results we will need for general posets follow from those in~\cite{DSV:ril}.  Pechenik~\cite{pec:map} also stated how the action of $\pro_K$ on increasing labelings of arbitrary posets is governed by the packed case.
 But we will give proofs  for $m$-packed labelings to keep this work self-contained.
\begin{prop}[{\cite[Lemma 3.9]{DSV:ril}}]
For any poset $P$ the map $\pro_K:\cL_m(P)\ra\cL_m(P)$  is a bijection.
\end{prop}
\begin{proof}
It suffices to construct the inverse function.  Given $L\in\cL_m(P)$, add one to all the labels and remove the label $m+1$ everywhere it occurs.  Now for any pair 
$x\lt y$ where $L(x)=m$ and $y$ is unlabeled, we move the label $m$ to the unlabeled element.  This process is iterated using the labels $m-1,\ldots,2$ at which point all unlabeled elements will be minimal.  Label these minimal elements with $1$. This is a step-by-step reversal of the algorithm used to define $\pro_K$ and so is its inverse.
\end{proof}

By the previous proposition, $\pro_K$ induces a group action on $\cL_m(P)$.  The current work will investigate its order, $o(\pro_K)$, and orbit structure.  
Of course, $o(\pro_K)$ depends on $m$ even though the notation does not reflect this.  But context will always make the value of $m$ clear.
In the next section we will collect various results that hold for wide classes of posets.  For example, we show in Theorem~\ref{prin} that if $P$ contains a certain type of chain then $o(\pro_K)$ is divisible by $m-1$.  We also find, using toggles, that for certain posets there is an equivariant bijection between the action of $\pro_K$ on $\cL_m(P)$ and the inverse of rowmotion on a subset of the ideals of $P$;  see Theorem~\ref{ProRhoEq}.  The rest of the paper will consider the  action of $\pro_K$ on various rooted trees.  
A {\em rooted tree} is a poset, $T$, with a minimum element $\zh$ such the Hasse diagram of $T$ is a graph-theoretic tree.  In general, we will apply graph theory terminology to a poset via its Hasse diagram.
Section~\ref{es} is dedicated to extended stars which are posets formed by identifying the minimum elements of a number of chains.  We give the precise value of $o(\pro_K)$ for arbitrary stars and any $m$ in Theorem~\ref{AllSta}.  When all the chains have the same length, we give more detailed information about the orbit structure in Corollary~\ref{bSta}.
In Section~\ref{cz} we investigate combs (which are chains with certain maximal elements added) and zippers (which are formed by pasting together combs).  For example, let $C_n$ be the comb formed from the chain with $n$ elements.  We show in Theorem~\ref{C:2n} and Proposition~\ref{C:n+1Z:n+2} that every orbit of $\pro_K$ acting on $\cL_m(C_n)$ has the same size if $m=2n$ and $m=n+1$, respectively.  Section~\ref{tlt} is devoted to trees with three maximal elements having a certain configuration.  If the tree $T$ has $n$ elements, then we completely describe the orbits on $\cL_m(T)$ when $m=n-k$ for $k\in[0,2]$ in Theorem~\ref{TltThm}.  The last section is dedicated to questions and directions for future research.

\section{General posets}

We prove various results about $\pro_K$ as applied to a wide variety of $m$-packed labelings of posets.  We begin by bounding the values of $m$ which can be used.
To state the result, we need to define
the {\em length} of a chain $C$ in a poset $P$ to be $|C|-1$, and
the {\em height} of $P$ as 
$$
h(P) = \text{ length of a longest chain in $P$}.
$$
The next proposition is related to the notion of consistent labelings given in~\cite[Definition 2.1 and Lemma 2.18]{DSV:ril}.
\begin{prop}
\label{ht}
For any poset $P$, an  $m$-packed labeling of $P$ exists if and only if
$$
h(P)+1 \le m \le \#P.
$$
\end{prop}
\begin{proof}

(a)  Throughout this proof, let $C$ be a longest chain of $P$ so that $\#C=h(P)+1$. For the forward direction, suppose there exists $L\in\cL_m(P)$.
Then the elements of $C$ must all receive different labels so that $m\ge \#C$.  And we have noted previously that $m\le\#P$.

For the converse, we will first suppose that $m=\#C$ and construct an $m$-packed labeling $L$  of $P$.  Let $P_1$ be the set of minimal elements of $P$ and label all these elements $1$.  Note that if $x$ is the minimal element of $C$ then $x\in P_1$ since, if not, then $x$ is above some element of $P_1$ and so $C$ can be made longer by appending this element.  Now label all the minimal elements of $P_2= P- P_1$ with $2$.  Continuing in this way we obtain subposets $P_1,\ldots,P_m$ where the labeling $L(x) = i$ for all $x\in P_i$ is $m$-packed, concluding this case.

Now suppose that $m<\#P$ and that, by induction on $m$, we have constructed an $m$-packed labeling $L$ of $P$.  We will construct an $(m+1)$-packed labeling $L'$ of the same poset.  Let $P_1,\ldots,P_m$ be the subposets of $P$ defined by $L(x) = i$ for all $x\in P_i$ and $i\in[m]$.  Since $m<\#P$, there is a $P_i$ with $\#P_i\ge 2$.  Pick some  $y\in P_i$.  We now define a labeling $L'$ of $P$ by
$$
L'(x) =
\begin{cases}
L(x)     &\text{if $x\in P_j$ for $j<i$ or $x\in P_i-\{y\}$},\\
i+1         &\text{if $x=y$},\\
L(x)+1   &\text{if $x\in P_j$ for $j>i$}.\\
\end{cases}
$$
It is easy to check that this is an $(m+1)$-packed labeling as desired.
\end{proof}

We will now prove some results which will show how certain subposets of $P$ will affect 
the orbit structure of $\pro_K$.  It will sometimes be best to express the action of $\pro_K$ in terms of toggles.   
An analogous presentation for the increasing labeled poset generalization of $\pro_K$ was given in~\cite{DSV:ril} (and in prior work~\cite[Proposition 2.5]{DPS:rop} in the increasing tableau case).
This is a $K$-theoretic analogue of the description of ordinary promotion in terms of Bender-Knuth toggles, see~\cite[Definition 13]{str:dac}.
    For $L\in\cL_m(P)$ and $1\leq i\leq m-1$, the $i$th {\em $K$-promotion toggle} $s_i$ acts on $L$ by setting, for each $x\in P$,
    \[
        s_i(L)(x)=
        \begin{cases}
            i+1 \hspace*{-8pt}&\mbox{if } L(x)=i \mbox{ \& the resulting labeling is still increasing},\\
            i   \hspace*{-8pt}&\mbox{if } L(x)=i+1 \mbox{ \&  the resulting labeling is still increasing},\\
            L(x) \hspace*{-8pt}&\mbox{otherwise.}\\
        \end{cases}
    \]

See Figure~\ref{ProTogFig} for an example. 

\begin{figure}
    \centering
\begin{tikzpicture}
\fill(2,0) circle(.1);
\fill(0,1) circle(.1);
\fill(2,1) circle(.1);
\fill(4,1) circle(.1);
\fill(0,2) circle(.1);
\fill(2,2) circle(.1);
\fill(4,2) circle(.1);
\fill(2,3) circle(.1);
\draw(2,-.5) node{$1$};
\draw(-.5,1) node{$3$};
\draw(1.5,1) node{$2$};
\draw(4.5,1) node{$2$};
\draw(-.5,2) node{$5$};
\draw(1.5,2) node{$4$};
\draw(4.5,2) node{$3$};
\draw(2,3.5) node{$6$};
\draw (2,0)--(0,1)--(0,2)--(2,3)--(4,2)--(4,1)--(2,0)--(2,3) (0,1)--(2,2)--(4,1);
\draw(6,1.5) node{$s_2$};
\draw(6,1) node{$\mapsto$};
\begin{scope}[shift={(8,0)}]
\fill(2,0) circle(.1);
\fill(0,1) circle(.1);
\fill(2,1) circle(.1);
\fill(4,1) circle(.1);
\fill(0,2) circle(.1);
\fill(2,2) circle(.1);
\fill(4,2) circle(.1);
\fill(2,3) circle(.1);
\draw(2,-.5) node{$1$};
\draw(-.5,1) node{$2$};
\draw(1.5,1) node{$3$};
\draw(4.5,1) node{$2$};
\draw(-.5,2) node{$5$};
\draw(1.5,2) node{$4$};
\draw(4.5,2) node{$3$};
\draw(2,3.5) node{$6$};
\draw (2,0)--(0,1)--(0,2)--(2,3)--(4,2)--(4,1)--(2,0)--(2,3) (0,1)--(2,2)--(4,1);
\end{scope}
\end{tikzpicture}
    \caption{An example of the $K$-promotion toggle $s_2$}
    \label{ProTogFig}
\end{figure}
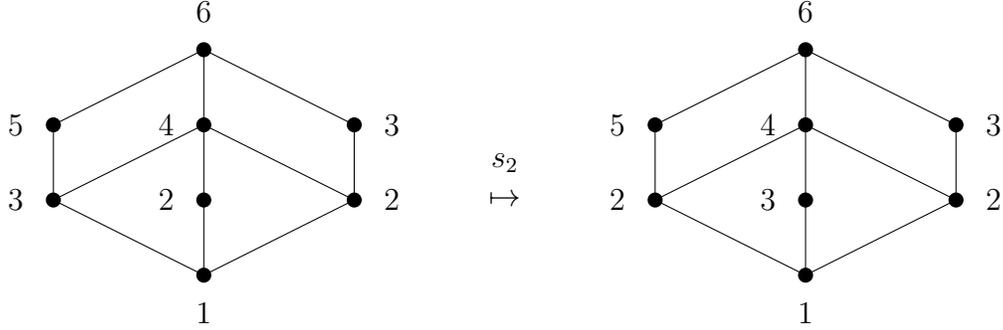

\begin{thm}[{\cite[Theorem 3.8]{DSV:ril}}]
\label{s_iThm}
    \label{ProTogThm}
    The action of $\pro_K$ on $L\in\cL_m(P)$ is given by

\vs{10pt}
    
    \eqqed{\pro_K(L)=s_{m-1}s_{m-2}\cdots s_1(L).}
\end{thm}

From here on out, all of our  posets $P$ will have a unique minimal element denoted  $\zh$.  
If $P\supset\{\zh\}$ then $\zh$ is covered by one or more elements.
We wish to show that if $\zh$ is covered by a single element, then $\pro_K$ acting on $\cL_m(P)$ has the same orbit structure as its action on $\cL_{m'}(P')$  for a certain $m'<m$ and subposet $P'$ of $P$.  We will do this by constructing an equivariant bijection.  But first, some definitions.

The {\em trunk}, $T$, of a poset $P$  having a $\zh$ is the longest chain
$$
\zh=x_1\lt x_2 \lt\ldots\lt x_t
$$
Such that each $x_i\in T$ is covered by a single element of $P$.  See Figure~\ref{tru:fig} for an example.

Write $G=\spn{g}$ to indicate that the group $G$ is generated by the element $g$.
Suppose $G=\spn{g}$ acts on $A$ and $H=\spn{h}$ acts on $B$ where $A,B$ are arbitrary sets.  So $g:A\ra A$ and $h:B\ra B$ are bijections.
The actions of $g$ on $A$ and $h$ on $B$ are {\em equivariant}, written $(g,A)\equiv(h,B)$, if there is a bijection $\phi:A\ra B$ such that
\beq
\label{equi:def}
h\phi = \phi g
\eeq
as maps from $A$ to $B$.  
It is well known that equivariant actions have the same orbit structure. 

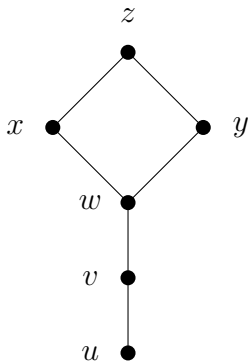
\begin{figure}
\begin{center}
\begin{tikzpicture}
\fill(0,0) circle(.1);
\fill(0,1) circle(.1);
\fill(0,2) circle(.1);
\fill(-1,3) circle(.1);
\fill(1,3) circle(.1);
\fill(0,4) circle(.1);
\draw(-.5,0) node{$u$};
\draw(-.5,1) node{$v$};
\draw(-.5,2) node{$w$};
\draw(-1.5,3) node{$x$};
\draw(1.5,3) node{$y$};
\draw(0,4.5) node{$z$};
\draw (0,0)--(0,2)--(-1,3)--(0,4)--(1,3)--(0,2);
\end{tikzpicture}
\end{center}
    \caption{The trunk of this poset $P$ is $T=\{u,v\}$}
    \label{tru:fig}
\end{figure}

\begin{prop}
\label{tru:pro}
 Let $P$ be a poset with a $\zh$. Consider the action of $\pro_K$ on $\cL_m(P)$. Suppose $P$ has trunk $T$ where we let $t=\#T$.
 Set $P'=P-T$ and $m'=m-t$.
 Then
 $$
 (\pro_K,\cL_m(P)) \equiv (\pro_K, \cL_{m'}(P')).
 $$
\end{prop}
\begin{proof}
Throughout we will add a prime to a notation for posets when it is applied to $P'$.
Let 
$$
T:\zh=x_1\lt x_2 \lt\ldots\lt x_t
$$
and suppose the element covering $x_t$ is $x_{t+1}$.
Then any $m$-packed labeling $L$ of $P$ has 
\beq
\label{tru:lab}
L(x_i)=i
\eeq
for $i\in [t+1]$.  It follows that the map
$\phi:\cL_m(P) \ra \cL_{m'}(P')$
given by $\phi(L)=L'$ where
\beq
\label{L'}
L'(x) = L(x) - t
\eeq
for $x\in P'$ is a well-defined bijection, i.e., $L'\in\cL_{m'}(P')$.
The fact that this bijection is equivariant now follows from Theorem~\ref{s_iThm} and the fact that, on $P$, the toggles $s_i$ for $i\in[t+1]$ act as the identity map.
\eprf

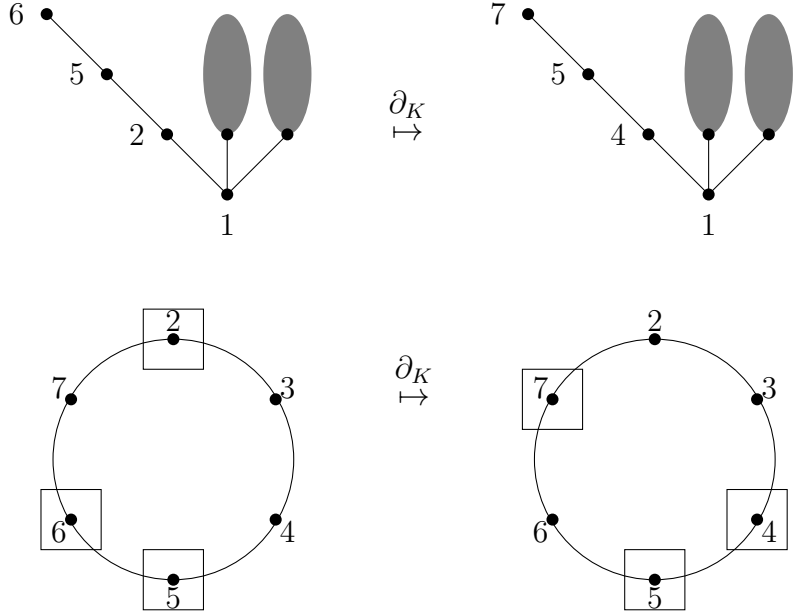
\begin{figure}
    \centering
\begin{tikzpicture}[scale=.8]
\fill[gray] (3,2) ellipse (.4 and 1);
\fill[gray] (4,2) ellipse (.4 and 1);
\fill(0,3) circle(.1);
\fill(1,2) circle(.1);
\fill(2,1) circle(.1);
\fill(3,0) circle(.1);
\fill(3,1) circle(.1);
\fill(4,1) circle(.1);
\draw (0,3)--(3,0)--(3,1) (3,0)--(4,1);
\draw(3,-.5) node{$1$};
\draw(1.5,1) node{$2$};
\draw(.5,2) node{$5$};
\draw(-.5,3) node{$6$};
\draw(6,1.5) node{$\pro_K$};
\draw(6,1) node{$\mapsto$};
\begin{scope}[shift={(8,0)}]
\fill[gray] (3,2) ellipse (.4 and 1);
\fill[gray] (4,2) ellipse (.4 and 1);
\fill(0,3) circle(.1);
\fill(1,2) circle(.1);
\fill(2,1) circle(.1);
\fill(3,0) circle(.1);
\fill(3,1) circle(.1);
\fill(4,1) circle(.1);
\draw (0,3)--(3,0)--(3,1) (3,0)--(4,1);
\draw(3,-.5) node{$1$};
\draw(1.5,1) node{$4$};
\draw(.5,2) node{$5$};
\draw(-.5,3) node{$7$};
\end{scope}
\end{tikzpicture}

\vs{20pt}

\hspace*{10pt}
\begin{tikzpicture}[scale=.8]
\draw(2,0) circle(2);
\fill(2,2) circle(.1);
\fill(3.7,1) circle(.1);
\fill(3.7,-1) circle(.1);
\fill(2,-2) circle(.1);
\fill(.3,-1) circle(.1);
\fill(.3,1) circle(.1);
\draw(2,2.3) node{$2$};
\draw(3.9,1.2) node{$3$};
\draw(3.9,-1.2) node{$4$};
\draw(2,-2.3) node{$5$};
\draw(.1,-1.2) node{$6$};
\draw(.1,1.2) node{$7$};
\draw(1.5,1.5) rectangle (2.5,2.5);
\draw(1.5,-1.5) rectangle (2.5,-2.5);
\draw(-.2,-.5) rectangle (.8,-1.5);
\draw(6,1.5) node{$\pro_K$};
\draw(6,1) node{$\mapsto$};
\begin{scope}[shift={(8,0)}]
\draw(2,0) circle(2);
\fill(2,2) circle(.1);
\fill(3.7,1) circle(.1);
\fill(3.7,-1) circle(.1);
\fill(2,-2) circle(.1);
\fill(.3,-1) circle(.1);
\fill(.3,1) circle(.1);
\draw(2,2.3) node{$2$};
\draw(3.9,1.2) node{$3$};
\draw(3.9,-1.2) node{$4$};
\draw(2,-2.3) node{$5$};
\draw(.1,-1.2) node{$6$};
\draw(.1,1.2) node{$7$};
\draw(-.2,.5) rectangle (.8,1.5);
\draw(3.2,-1.5) rectangle (4.2,-.5);
\draw(1.5,-1.5) rectangle (2.5,-2.5);
\end{scope}
\end{tikzpicture}
\caption{$K$-promotion acts cyclically on the  branch $B=\{2,5,6\}$}
        \label{CycPath}
\end{figure}

Given a poset $P$ with $\zh$, the {\em principal subposets} of $P$ are the connected components of $P-\zh$.  A principal subposet which is a chain is called a {\em branch}.  If $a,b\in\bbN$ then  $\gcd(a,b)$ will denote their greatest common divisor and we write $a\mid b$ if $a$ divides evenly into $b$.
\begin{thm}
\label{prin}
	Let $P$ be a poset with $\zh$ having a branch on $k$ vertices, where $1\leq k\leq m-2$.  Consider the action of $\pro_K$ on $\cL_m(P)$.
    \begin{enumerate}[label=\textup{(\alph*)}]
        \item We have 
        $$
        m-1\mid o(\pro_K).
        $$
        \item If $\gcd(k,m-1)=1$, then $m-1$ divides the size of any orbit of the action. 
    \end{enumerate}
\end{thm}
\begin{proof}
(a)  Let $B$ be a $k$-vertex  branch of $P$, and $L\in\cL_m(P)$. We study the 
effect on $B$ induced by computing $\pro_K L$. 
See, for example, the top line  in Figure~\ref{CycPath} where  $m=7$ and $k=3$.
Consider the labels $2,3,\ldots,m$ written clockwise around a circle with the $k$ vertices of $B$ represented as boxes placed around their labels in $L$.  See the bottom line in Figure~\ref{CycPath}. Then $\pro_K$ simply rotates the boxes one step counterclockwise around the circle.  

    To finish the proof it suffices, since $o(\pro_K)$ is the least common multiple of all the orbit sizes,  to find an orbit $\cO$ of the action of $\pro_K$ on $\cL_m(P)$ such that $m-1\mid\#\cO$ .  Let $L\in\cL_m(P)$  be such that the induced labeling on $B$ is an interval of $k$ integers and let $\cO$ be the $\pro_K$ orbit containing $L$.  From the rotational description in the previous paragraph, the labels on this branch will first return to the given interval after $m-1$ applications of $\pro_K$.  It follows that 
    $m-1 \mid \#\cO$ as desired.

\medskip

(b) If $\gcd(k,m-1)=1$ then a similar argument to (a) shows that   
$m-1 \mid \#\cO$ no matter which $k$ integers appear on $B$.
\end{proof}

Combining  Proposition~\ref{tru:pro} and Theorem~\ref{prin} give the following result.
\begin{cor}
 Let $P$ be a poset with $\zh$ and trunk $T$ having $|T|=t$.  Suppose $P':=P-T$ satisfies the hypotheses of Theorem~\ref{prin}.
 \begin{enumerate}
     \item The order of $\pro_K$ on $P$ is divisible by $m-t-1$.
     \item If $\gcd(k,m-t-1)=1$, then $m-t-1$ divides the size of any orbit on $P$.\hqed
 \end{enumerate}
\end{cor}

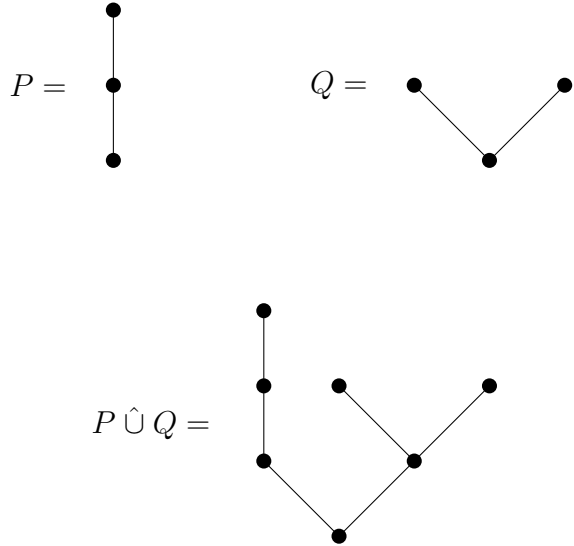
\begin{figure}
    \centering
\begin{tikzpicture}
\fill(0,0) circle(.1);
\fill(0,1) circle(.1);
\fill(0,2) circle(.1);
\draw (0,0)--(0,2);
\draw(-1,1) node{$P=$};
\begin{scope}[shift={(5,0)}]
\fill(0,0) circle(.1);
\fill(-1,1) circle(.1);
\fill(1,1) circle(.1);
\draw (-1,1)--(0,0)--(1,1);
\draw(-2,1) node{$Q=$};
\end{scope}
\begin{scope}[shift={(3,-5)}]
\fill(-1,1) circle(.1);
\fill(-1,2) circle(.1);
\fill(-1,3) circle(.1);
\fill(0,0) circle(.1);
\fill(0,2) circle(.1);
\fill(1,1) circle(.1);
\fill(2,2) circle(.1);
\draw (-1,3)--(-1,1)--(0,0)--(2,2) (1,1)--(0,2);
\draw(-2.5,1.5) node{$P\bdu Q=$};
\end{scope}
\end{tikzpicture}
    
    \caption{The bounded union of posets $P$ and $Q$}
    \label{oplusFig}
\end{figure}

Because of Proposition~\ref{tru:pro}, we will only consider posets having an empty trunk in what follows.  The orbit information for a poset with a trunk can easily be determined from the corresponding trunkless poset and the previous corollary.

We wish to define a way of combining posets which behaves well with respect to $\pro_K$.
If $P$ and $Q$ are disjoint posets then their {\em bounded union} is
$$
P \bdu Q = \text{ $P\uplus Q$ with a minimum element $\zh$ adjoined}
$$
where $\uplus$ is disjoint union.  An example of this construction is given in Figure~\ref{oplusFig}.  Also, if $P$ is any poset with $\zh$ and $\cO_1,\ldots,\cO_k$ are the orbits of $\pro_K$ on the $\cL_m(P)$ then we will consider the multiset
\beq
\label{omP}
o_m(P) =\{\{\#\cO_1,\ldots,\#\cO_k\}\}.
\eeq
Finally, we use $\lcm$ for least common multiple.
\begin{prop}
\label{oplusProp}
Let $P,Q$ be posets each with a $\zh$ and satisfying $h(P)=h(Q):=h$.  Suppose that $m=h+1$ with $o_m(P) = \{\{c_1,\ldots,c_k\}\}$ and $o_m(Q)=\{\{d_1,\ldots,d_l\}\}$.  Then
$$
o_{m+1}(P\bdu Q) =\{\{\lcm(c_i,d_j) \mid 1\le i\le k,\ 1\le j\le l\}\}.
$$
\end{prop}
\begin{proof}
Define a map 
$$
\phi:\cL_{m+1}(P\bdu Q)\ra\cL_m(P)\times\cL_m(Q)
$$
by sending $L\mapsto(L_P,L_Q)$ such that
\begin{align*}
L_P(x) & = L(x) -1\text{ if $x\in P$},\\
L_Q(y) & =  L(y) -1\text{ if $y\in Q$}.
\end{align*}
It is easy to see that the hypothesis $m=h(P)+1=h(Q)+1$ implies the map is well defined in that $L_P$ and $L_Q$ are $m$-packed labelings of their respective posets.  It is also a simple matter to construct an inverse for $\phi$, so the map is bijective.

Note also that we have the equivariance
$$
\phi\pro_K(L)=\pro_K\phi(L)
$$
where $\pro_K$ is applied component-wise to $\phi(L)$.  This is because the hypotheses on $m$ and $h$ force $L(\zh_P)=L(\zh_Q)=2$ for all $m$-packed labelings $L$.  Thus the action of $\pro_K$ on $P$ and $Q$ as subposets of $P\bdu Q$ is the same, up to subtracting $1$, as its action on $P$ and $Q$ themselves.  Clearly, if $L_P$ is in an orbit of size $c$ and $L_Q$ is in an orbit of size $d$ then $(P,Q)$ is in an orbit of size $\lcm(c,d)$.  The proposition now follows from the equivariance.
\end{proof}

Rowmotion is another action on posets which has received much attention.  Promotion and rowmotion are often connected for certain posets.  See, for example, the article of Striker and Williams~\cite{SW:pr}.  We wish to show that this is also true for $m$-packed labelings of certain posets $P$ when $m=h(P)+2$.  
In fact, in~\cite{DSV:ril} it is shown that there is an equivariant bijection between $K$-promotion on increasing labelings of any poset and rowmotion on a corresponding object which they call the gamma poset. In the case we are studying, the gamma poset is particularly simple.  
We first recall the definition of rowmotion.

A {\em (lower order) ideal} of $P$ is $I\sbe P$ such that $x\in I$ and 
$y\le x$ implies $y\in I$.  We let
$$
\cJ(P) = \{ I \mid \text{$I$ is an ideal of $P$}\}.
$$
If $Q\sbe P$ then the {\em ideal generated by $Q$} is
$$
Q\da\ = \{y \mid \text{$y\le x$ for some $x\in Q$}\}.
$$
The {\em rowmotion operator} is a map 
$\rho:\cJ(P)\ra\cJ(P)$
defined by 
$$
\rho(I)= Q\da
$$ 
where $Q$ is the set of minimal elements of $P-I$.
To connect $\rho$ with $\pro_K$, let $P$ be a poset with a $\zh$ such that all maximal chains of $P$ have length $h(P)$.
Note that this implies that $P$ is {\em ranked} in that, for any $x\in P$, all maximal chains from $\zh$ to $x$ have the same length.  This common length is called the {\em rank} of $x$ and denoted $\rk x$.
Consider the $m$-packed labelings of $P$ where $m =h(P)+2$.  So, every maximal chain of $P$ will be missing exactly one label.  It follows that if $L\in\cL_m(P)$ and $x\in P$ then either 
$L(x)=\rk x + 1$ or $L(x) = \rk x + 2$.  Furthermore
$$
I(L) =\{ x \mid L(x) =\rk x + 1\}
$$
is a lower ideal of $P$.   Then every $I\in\cJ(P)$ is of the form $I=I(L)$ for some $m$-packed labeling $L$ except those of the form
\beq
\label{equirk}
I = \{x \in P \mid \rk x\le k\}
\eeq
for some $k$ since these would come from a labeling where all maximal chains are missing the same label, and such a labeling is not $m$-packed.  Let
$$
\cJ'(P) = \{ I \mid \text{$I\in\cJ(P)$ but $I$ is not of the form~\eqref{equirk}}\}.
$$
It is not hard to see from what we have done above that the map $L\mapsto I(L)$ is a bijection
\beq
\label{piDef}
\pi:\cL_m(P)\ra\cJ'(P).
\eeq

Now consider the rowmotion operator $\rho$.  Note that the ideals in~\eqref{equirk} are all in an orbit of $\rho$.  So $\rho$ restricts to an action on $\cJ'(P)$. In fact, we will show that the map $\pi$ is an equivariant bijection between $K$-promotion and the inverse of rowmotion. To prove this, we will need to express our group actions in terms of toggles.

We first recall the toggle operator presentation of rowmotion on order ideals.  
    For a poset $P$ and $x\in P$, the {\em toggle} $t_x$ acts on $\cJ(P)$ via
    \[
        t_x(I)=
        \begin{cases}
            I\uplus\{x\}\quad&\mbox{if } x\notin I \mbox{ and } I\uplus\{x\}\in \cJ(P),\\
            I-\{x\}   \quad&\mbox{if } x\in I \mbox{ and } I-\{x\}\in \cJ(P),\\
            I         \quad&\mbox{otherwise.}\\
        \end{cases}
    \]
A {\em linear extension} of $P$ is a list of $P$'s elements $x_1,\ldots,x_n$ such that $x_i<_P x_j$ implies $i<j$,

\begin{thm}[{\cite{CF:oar}}]
    \label{RowTogThm}
    The action of $\rho$ on $\cJ(P)$ is equal  to the product of toggles $\{t_x\mid x\in P\}$ in reverse order of any linear extension of $P$.\hqed
\end{thm}

Next, we explore how $\pi$ commutes with toggles.
\begin{lem}[{\cite[Lemma 4.32]{DSV:ril}}]
    \label{ProRowTog}
    Let $P$ be a poset with $\zh$ such that all maximal chains have the same length, and set $m=h(P)+2$. Then for any $L\in \cL_m(P)$ and $i<m$,
    \[\pi (s_i(L)) = \left(\prod_{\rk z=i-1}t_z\right)\pi(L).\]
\end{lem}
\begin{proof}
 Let 
    \[A=\pi(s_i(L)) \quad \mbox{and}\quad B=\left(\prod_{\rk(z)=i-1}t_z\right)\pi(L).\]
    We wish to show $A=B$.  Note that the only $x\in P$ whose labels can change when applying $s_i$ are those with $L(x)=i$ or $i+1$. 
    Also, the choice of $m$ implies that $\rk(x) = L(x)-1$  or $L(x)-2$. 
    So, the proof breaks into four cases depending on label and rank.

    \emph{Case 1: $L(x)=i$ and $\rk(x)=i-1$.}\\
    First suppose that there exists $y$ covering $x$ with $L(y)=i+1$. Then $x,y\in \pi(L)$, so $x\in t_x(\pi(L))$. Consequently $x\in B$. On the other hand, both $s_iL(x)=i$ and $s_i L(y) = i+1$ still hold.  So $x\in A$ as well. 

    Now, suppose there is no such $y$. Then every $y$ covering $x$ must have label $i+2$, so $x\in \pi(L)$ is maximal. Consequently $x\notin t_x(\pi(L))$, so $x\notin B$. On the other hand, \[s_iL(x)=i+1>\rk(x)+1,\] 
    so $x\notin A$ as well.
    
    \emph{Case 2: $L(x)=i$ and $\rk(x)=i-2$.}\\
    Since $L(x)>\rk(x)+1$, it follows that $x\notin \pi(L)$. The toggles applied in $B$ are to a different rank set in $P$, so also $x\notin B$. By the choice of $m$ and the fact that $i<m$ we have that $x$ is not maximal and every $y$ covering $x$ has $L(y)=i+1$, so $s_iL(x)=i$. Consequently, $x\notin A$ also.
    
    \emph{Case 3: $L(x)=i+1$ and $\rk(x)=i$.}\\
    This case is easy to check if $i=0$, so assume $i>0$.
    Since $L(x)=\rk(x)+1$, it follows that $x\in \pi(L)$. By assumption, $x$ lies on a maximal chain of the uniform length and $\rk(x)>0$, so $x$ covers at least one element and all such elements have  label $i$. Thus $x\in A$. On the other hand, the toggles in $B$ are acting on a different rank, so $x\in B$ as well.
    
    \emph{Case 4: $L(x)=i+1$ and $\rk(x)=i-1$.}\\
    First suppose that there exists $y$ covered by $x$ with $L(y)=i$. Then $x,y\notin \pi(L)$, so $x\notin t_x(\pi(L))$. Consequently $x\notin B$. On the other hand, we still have $s_i(L)(x)=i+1$, so $x\notin A$ as well. 

    Now, suppose there is no such $y$. Then every $y$ covered by $x$ must have label $i-1$, and so lies in $\pi(L)$. But $x\notin\pi(L)$, so $x$ is minimal among elements not in $\pi(L)$. Consequently $x\in t_x(\pi(L))$, so $x\in B$. On the other hand, $s_iL(x)=i=\rk(x)+1$,
    so $x\in A$ too.
\end{proof}

We now have everything in place to prove the desired equivariance.
\begin{thm}[{\cite[Theorem 4.21]{DSV:ril}}]
\label{ProRhoEq}
Let $P$ be a poset with $\zh$ such that all maximal chains have the same length. If $m=h(P)+2$ 
then the map $\pi$ in~\eqref{piDef} is an equivariant bijection, showing that
$$
(\pro_K,\cL_m(P)) \equiv (\rho^{-1},\cJ'(P)).
$$
\end{thm}
\begin{proof}
We have already noted that $\pi$ is a bijection.  So we just need to show that for any $L\in\cL_m(P)$ we have
$$
\pi\circ\pro_K(L) = \rho^{-1}\circ\pi(L).
$$
Let the elements of  $P$  be naturally labeled with $\{1,2,\ldots,n\}$ along ranks, from lowest to highest. In particular we have that $1,2,\ldots,n$ is a linear extension. By Theorem \ref{RowTogThm} and Theorem \ref{ProTogThm}, the actions of $K$-promotion and inverse rowmotion are given by
$$
\pro_K(L) = s_{m-1}s_{m-2}\cdots s_1(L) \qquad\mbox{and}\qquad \rho^{-1}(I)=t_nt_{n-1}\cdots t_1(I)
$$
where we can take $I=\pi(L)$.
Now commutativity follows  from repeated applications of Lemma \ref{ProRowTog}.
\end{proof}

\section{Extended stars}
\label{es}

In graph theory, a {\em star} is the complete bipartite graph $K_{1,n}$.  For us, 
an {\em extended star}, $S$, will be  a  poset with $\zh$ such that $S-\zh$ is a disjoint union of chains.  Recall that such chains are called branches.  We write 
$S=S(b_1,b_2,\ldots,b_k)$ if $S$ has $k$ branches with the $i$th branch $B_i$ having $b_i$ elements.  Interestingly, it follows from the next result that the order of $K$-promotion on $\cL_m(S)$  is almost always $m-1$ and so independent of $\#S$ in such cases.
\begin{thm}
\label{AllSta}
Let $S=S(b_1,b_2,\ldots,b_k)$ and $m$ be such that $S$ has an $m$-packed labeling.  Then
$$
o(\pro_K) =
\case{1}{if $b_i=m-1$ for all $i$,}{m-1}{else.}
$$
\end{thm}
\begin{proof}
If $b_i=m-1$ for all $i$, then there is only one $m$-packed labeling of $S$ and the result follows.

If there is a branch with $b_i\neq m-1$ then, by Proposition~\ref{ht}, we must have $b_i\le m-2$.  So Theorem~\ref{prin} (a) applies and $m-1\mid o(\pro_K)$.  To show that we have equality note that, as in the proof of Theorem~\ref{prin} (a), the induced action of $\pro_K$ on each branch $B_j$ is just rotation of $b_j$ labels around a circle of $m-1$ elements. Clearly, after $m-1$ rotations the labels on $B_j$ will have returned to their original state.  So $m-1= o(\pro_K)$.
\end{proof}

The next result will illustrate the use of Theorem~\ref{ProRhoEq}. 
To state part of it, we need to review some definitions about group actions and statistics.  Let $S$ be a finite set.  A {\em statistic} on $S$ is a function
$$
\st:S\ra\bbN.
$$
We can apply $\st$ to subsets $T\sbe S$ by letting
$$
\st T = \sum_{t\in T} \st t.
$$
Now suppose $G$ is a finite group acting on $S$.  We call $\st$ {\em homomesic} if there is a constant $c$ such that for any orbit $\cO$ of the action
\beq
\label{HomEq}
\frac{\st\cO}{\#\cO}=c.
\eeq
Otherwise put, the average value of $\st$ is the same over all orbits.  Homomesy was first defined by Propp and Roby~\cite{PR:hpt} and has since become a much-studied property in dynamical algebraic combinatorics.

Let $P$ be a ranked poset and let $L$ be an $m$-packed labeling of $P$.  Call $x\in P$ {\em minimally labeled} if $L(x)=\rk x + 1$.  Note that such $x$ are exactly the ones whose label is as small as possible given the fact that $L$ is increasing on $P$.  Let
$$
\cM(L) = \#\{ x\in P \mid \text{$x$ is minimally labeled in $L$}\}.
$$ 
The following properties were proven for rowmotion on extended stars in the paper of Dang\-wal et al.~\cite[Theorem 3.1]{DKLLSS:rrt} by employing a tiling model.  So, using Theorem~\ref{ProRhoEq}, we immediately get the following facts about $K$-promotion on such stars where all $k$ branches have the the same number of elements $b$, denoted by $b^k$.
\begin{cor}
\label{bSta}
Consider the star $S=S(b^k)$ and let $m = b+2$.  Then the orbits of $\pro_K$ acting in $\cL_m(S)$ satisfy the following.
\ben
\item[(a)]  Every orbit $\cO$ has $\#\cO=b+1$.
\item[(b)]  The number of orbits is $(b+1)^{k-1}-1$.
\item[(c)]  For any orbit $\cO$ we have
$$
\cM(\cO) = b+1 + k\binom{b+1}{2}.
$$
It follows that the statistic $\cM$ is homomesic.\hqed
\een
\end{cor}

\section{Combs and Zippers}
\label{cz}

The {\em comb}, $C_n$, is formed from a chain with $n$ elements by adding a new maximal element attached to each of the vertices of the chain.  We call the original chain the {\em spine} of the comb.  A labeling of the comb $C_3$ is displayed on the left in Figure~\ref{L^(3)}.  Combs appear in various contexts such as a tree analogue of the Robinson-Schensted-Knuth correspondence for Young tableaux~\cite{SY:pat,sta:fl}.  We also define the {\em zipper poset}
$$
Z_n = C_n \bdu C_n.
$$

We begin with a simple observation about the orbit sizes.
\begin{prop}
If $n+1\le m\le 2n$  and $\cO$ is any orbit of $\pro_K$ on $\cL_m(C_n)$ then
     $$
     (m-1) \mid \#\cO.
     $$
\end{prop}
\begin{proof}
Note that $C_n$ has a branch with one element.  So this result follows immediately from Theorem~\ref{prin} (b).
\end{proof}

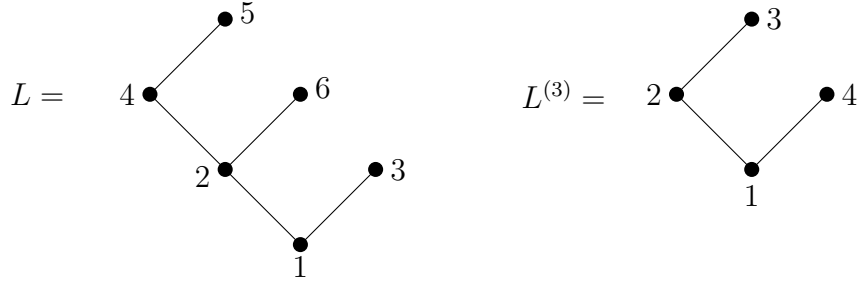
\begin{figure}
\centering
\begin{tikzpicture}
\fill(0,0) circle(.1);
\draw(0,-.3) node{$1$};
\fill(-1,1) circle(.1);
\draw(-1.3,.9) node{$2$};
\fill(1,1) circle(.1);
\draw(1.3,1) node{$3$};
\fill(-2,2) circle(.1);
\draw(-2.3,2) node{$4$};
\fill(0,2) circle(.1);
\draw(.3,2.08) node{$6$};
\fill(-1,3) circle(.1);
\draw(-.7,3.08) node{$5$};
\draw (1,1)--(0,0)--(-2,2)--(-1,3) (-1,1)--(0,2);
\draw(-3.5,2) node{$L=$};
\begin{scope}[shift={(6,1)}]
\fill(0,0) circle(.1);
\draw(0,-.35) node{$1$};
\fill(-1,1) circle(.1);
\draw(-1.3,1) node{$2$};
\fill(1,1) circle(.1);
\draw(1.3,1) node{$4$};
\fill(0,2) circle(.1);
\draw(.3,2) node{$3$};
\draw (1,1)--(0,0)--(-1,1)--(0,2);
\draw(-2.5,1) node{$L^{(3)}=$};
\end{scope}
\end{tikzpicture}
    \caption{A labeling $L\in\cL_6(C_3)$ and the induced labeling $L^{(3)}\in\cL_4(C_2)$}
    \label{L^(3)}
\end{figure}

If $m=2n=\# C_n$, its maximum value, much more can be said.
For any odd integer $2n+1$ we define the corresponding {\em double factorial}
$$
(2n+1)!! = (2n+1)(2n-1)(2n-3)\cdots 1.
$$
We will also need a formula for the number of $(\#T)$-packed labelings of a tree $T$.
An $m$-packed labeling with $m=\#T$ is call a {\em standard labeling} of $T$ and we let
\beq
\label{f^T}
f^T = \text{ the number of standard labelings of $T$}.
\eeq
The {\em hook length} of  $x\in T$ is
$$
h_x = \#\{ y \mid y\ge x\},
$$
that is, the size of the upper order ideal generated by $x$.  The following result is known as the Hook Length Formula for Trees and is due to Knuth~\cite{knu:acp3}
\begin{equation}
\label{hook}
    f^T = \frac{n!}{\prod_{x\in T} h_x}
\end{equation}
where $n=\#T$.  We now return to the case of the comb.
\begin{thm}
\label{C:2n}
Every orbit of $\pro_K$ on $\cL_{2n}(C_n)$ has  the same size.  And for any such orbit $\cO$ we have
$$
\#\cO \mid (2n-1)!!.
$$
\end{thm}
\begin{proof}
 We induct on $n$ where the result is easy to check for $n\le2$.  So, assume the result for $n$ and suppose $L\in\cL_{2n+2}(C_{n+1})$.   Let $y$ be the maximal element of $C_{n+1}$ covering $\zh$ and let $L(y)=l$.  From $L$ we can construct a labeling
  of $C_n$ by first removing $\zh$ and $y$ and then replacing the remaining 
  labels $\{2,\ldots,l-1,l+1,\ldots, 2n+2\}$ with the labels $[2n]$ in an order-preserving way.  The resulting labeling will be denoted $L^{(l)}$ and will be called the {\em restriction} of $L$.  An example can be found in Figure~\ref{L^(3)}.

 Now consider any orbit $\cO$ for the action on $\cL_{2n+2}(C_{n+1})$.  From Theorem~\ref{prin} (b) we have that $(2n+1)|\#\cO$.  And, as in the proof of that theorem, if $L\in\cO$ with $L(y)=l$ then the successive labels of $y$ in $\cO$ are $l,l-1,l-2,\ldots$ where these numbers are taken modulo $2n+1$ using representatives in the interval $[2,2n+2]$.  So we can pick, without loss of generality, $\cO$ to start with a labeling $L_1$ such that $L_1(y) = 2n+2$.  It will also be convenient to list the labelings in $r$ rows of length $2n+1$ so that 
 \beq\label{Orn}
 \#O = r(2n+1).
\eeq

 Now consider the first row of $\cO$ with labelings $L_1,L_2,\ldots,L_{2n+1}$ in order.
 Let $M_1=L_1^{(2n+2)}$ and consider the orbit $\cP\sbe\cL_{2n}(C_n)$ containing $M_1$. By induction, we can write as $\cP=(M_1,M_2,\ldots,M_{N})$ where 
 $N\mid (2n-1)!!$.  We claim that for all $1\le i\le 2n+1$ we have
 \beq\label{LiMi}
 L_i^{(2n+3-i)} = M_i.
 \eeq
 We induct on $i$ where the case $i=1$ is true by assumption.  Now suppose $1\le i\le 2n$ and that the desired equation holds for $L_i$.  By the bounds on $i$, we have  $L_i(z)=2$, where $z$ is the element of the spine covering $\zh$.
 So, computing $\pro_K L_i$ begins by moving the $2$ on $z$ down to $\zh$.  And from there on out, the path of elements  moved is the same as in computing $\pro_K M_i$ by the induction assumption and the fact that these moves only depend on the relative sizes of the elements.  It follows that equation~\eqref{LiMi} continues to hold at $i+1$.

 Now consider what happens when moving from $L_{2n+1}$ to $L_{2n+2}$ in the next row.  We claim that
 $$
 L_{2n+2}^{(2n+2)} = M_{2n+1}.
 $$
 Indeed, in $L_{2n+1}$ we have $L_{2n+1}(y) = 2$.  So applying $\pro_K$ just moves the $2$ from $y$ down to $\zh$ without disturbing any of the elements of $C_n-\{\zh,y\}$.
 So $L_{2n+2}^{(2n+2)}=L_{2n+1}^{(2)}=M_{2n+1}$ by~\eqref{LiMi}.

 By the same considerations, the restrictions of the labelings in the second row will given by the labelings
 $M_{2n+1}, M_{2n+2},\ldots,M_{4n+1}$ where the subscripts are taking modulo $N=\#\cP$.  Then the third row restrictions will start $M_{4n+1},M_{4n+2},\ldots$ and so forth.
 It follows that if we consider the restrictions of the elements in $\cO$ but ignore the last restriction in each row of $2n+1$ elements, then the result will be a concatenation of copies of the orbit $\cP$.  Recalling our notation $r$ for the number of rows, it follows that $2nr$ is a multiple of $\#\cP=N$.  In fact, $r$ must be the smallest positive integer such that $N\mid 2nr$.  And $\#\cO=(2n+1)r$.  But, by induction, the characterization of $r$ in terms of divisibility just given does not depend on which orbit $\cO$ is chosen.  This proves that the orbit size is constant over all orbits.

For the divisibility result it suffices to show that $\#\cL_{2n}(C_n)=(2n-1)!!$ since each orbit has the same size.  But, since 
$m=2n=\# C_n$, we are counting natural labelings. So, by formula~\eqref{hook} applied to $T=C_n$, we have
$$
\#\cL_{2n}(C_n) = \frac{(2n)!}{(2n)(2n-2)\cdots 2\cdot 1^n} = (2n-1)!!
$$
as desired.
\end{proof}

We have just characterized the orbit sizes for $\pro_K$ acting on $\cL_m(C_n)$ where, by Proposition~\ref{ht}, $m$ is as large as possible.  We can do the same when $m$ is at the opposite end of the spectrum.

\begin{prop}
\label{C:n+1Z:n+2}
We have the following,
\ben
\item[(a)]  Every orbit of $\pro_K$ on $\cL_{n+1}(C_n)$ has size $\lcm(1,2,\ldots,n)$.
\item[(b)]  Every orbit of $\pro_K$ on $\cL_{n+2}(Z_n)$ has size $\lcm(1,2,\ldots,n)$.
\een
\end{prop}
\begin{proof}
(a)
    In $C_n$, consider the chain consisting of the spine of the comb and the topmost maximal element. This chain has $n+1=m$ elements, so it is forced to have labels $1,2,\ldots,m$. Let the labels on the $n$ maximal elements of the comb be  $t_1,t_2,\ldots,t_n$ listed from bottom to top. 
    By the increasing condition we have $i+1\le t_i\le m$ for all $i\in[n]$.
    Thus the action of $\pro_K$ independently decrements each label $t_i$ cyclically within $\{i+1,i+2,\ldots,m\}$. It follows that the order of $\pro_K$ acting on the $i$th maximal  element is 
$m-i$. Hence, the size of any $K$-promotion orbit on an $m$-packed labeling of $C_n$ is
    \[\lcm(m-1,m-2,\ldots,m-n)=\lcm(1,2,\ldots,n).\]

\medskip

(b)  This follows immediately from part (a), the fact that $Z_n=C_n\bdu C_n$, and Proposition~\ref{oplusProp}.
\end{proof}

\section{A tree with three leaves}
\label{tlt}

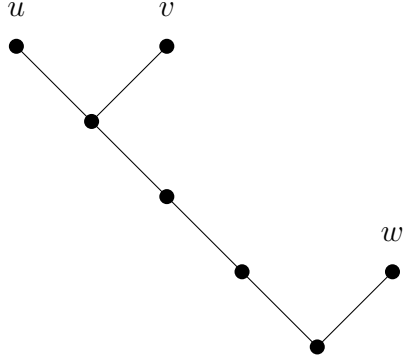
\begin{figure}
\hspace{130pt}
\begin{tikzpicture}
\fill(4,2) circle(.1);
\fill(3,3) circle(.1);
\fill(5,3) circle(.1);
\fill(2,4) circle(.1);
\fill(1,5) circle(.1);
\fill(0,6) circle(.1);
\fill(2,6) circle(.1); 
\draw (4,2)--(0,6) (1,5)--(2,6) (4,2)--(5,3);
\draw(0,6.5) node{$u$};
\draw(2,6.5) node{$v$};
\draw(5,3.5) node{$w$};
 \end{tikzpicture}
    \caption{The tree $T(3)$}
    \label{3Tree}
\end{figure}

A non-$\zh$ maximal element in a rooted tree will be called a {\em leaf}.  If a tree has  exactly  one leaf then  it is a chain.  And if it has exactly two then it is an extended star with two branches and  a trunk (which may be empty).  In the first case, the orbit structure of $\pro_K$ on $m$-packed chains is trivial.  And the second case can be analyzed with the help of Proposition~\ref{tru:pro} and Theorem~\ref{AllSta}.  We now consider a particular tree having three leaves.

Let $C$ be a chain of length $c$.  We now form a tree $T(c)$ by adding two leaves $u$ and $v$ covering the $\oh$ of the chain and one leaf $w$ covering the chain's $\zh$.  The Hasse diagram of $T(3)$ is displayed in Figure~\ref{3Tree}.  We will completely describe the orbits of $\pro_K$ on $\cL_m(T(c))$ for three values of $m$. 

\begin{thm}
\label{TltThm}
    Consider $T(c)$ having  $n:=\#T(c)= c+4$ elements. 
Then an $m$-packed labeling of $T(c)$ exists if and only if $n-2\le m \le n$.
In each of these cases, the orbits of $\pro_K$  on $\cL_m(T(c))$ can be described as follows:
    \begin{enumerate}[label=\textup{(\alph*)}]
        \item For $m=n-2$, there is one orbit of length $m-1$.
        \item For $m=n-1$:
            \begin{enumerate}[label=\textup{(\roman*)}]
                \item If $c$ is even, there are 3 orbits of length $m-1$.
                \item If $c$ is odd, there are 2 orbits, one of length $m-1$ and one of length $2(m-1)$.
            \end{enumerate}
        \item For $m=n$:
            \begin{enumerate}[label=\textup{(\roman*)}]
                \item If $c$ is even, there are 2 orbits of length $m-1$.
                \item If $c$ is odd, there is one orbit of length $2(m-1)$.
            \end{enumerate}
    \end{enumerate}
\end{thm}

\begin{figure}

\vspace*{-5pt}

$$
\cO_1 = 
\left(
\barr{cccc}

\begin{tikzpicture}[scale=.5]
\fill(3,0) circle(.1);
\fill(2,1) circle(.1);
\fill(1,2) circle(.1);
\fill(0,3) circle(.1);
\fill(4,1) circle(.1);
\fill(2,3) circle(.1);
\draw(3,-.5) node{$1$};
\draw(1.5,1) node{$2$};
\draw(.5,2) node{$3$};
\draw(0,3.5) node{$5$};
\draw(2,3.5) node{$4$};
\draw(4,1.5) node{$5$};
\draw (4,1)--(3,0)--(0,3) (1,2)--(2,3);
\draw (5,0) node{,};
\end{tikzpicture}
&
\begin{tikzpicture}[scale=.5]
\fill(3,0) circle(.1);
\fill(2,1) circle(.1);
\fill(1,2) circle(.1);
\fill(0,3) circle(.1);
\fill(4,1) circle(.1);
\fill(2,3) circle(.1);
\draw(3,-.5) node{$1$};
\draw(1.5,1) node{$2$};
\draw(.5,2) node{$3$};
\draw(0,3.5) node{$4$};
\draw(2,3.5) node{$5$};
\draw(4,1.5) node{$4$};
\draw (4,1)--(3,0)--(0,3) (1,2)--(2,3);
\draw (5,0) node{,};
\end{tikzpicture}
&
\begin{tikzpicture}[scale=.5]
\fill(3,0) circle(.1);
\fill(2,1) circle(.1);
\fill(1,2) circle(.1);
\fill(0,3) circle(.1);
\fill(4,1) circle(.1);
\fill(2,3) circle(.1);
\draw(3,-.5) node{$1$};
\draw(1.5,1) node{$2$};
\draw(.5,2) node{$3$};
\draw(0,3.5) node{$5$};
\draw(2,3.5) node{$4$};
\draw(4,1.5) node{$3$};
\draw (4,1)--(3,0)--(0,3) (1,2)--(2,3);
\draw (5,0) node{,};
\end{tikzpicture}
&
\begin{tikzpicture}[scale=.5]
\fill(3,0) circle(.1);
\fill(2,1) circle(.1);
\fill(1,2) circle(.1);
\fill(0,3) circle(.1);
\fill(4,1) circle(.1);
\fill(2,3) circle(.1);
\draw(3,-.5) node{$1$};
\draw(1.5,1) node{$2$};
\draw(.5,2) node{$3$};
\draw(0,3.5) node{$4$};
\draw(2,3.5) node{$5$};
\draw(4,1.5) node{$2$};
\draw (4,1)--(3,0)--(0,3) (1,2)--(2,3);
\end{tikzpicture}

\earr
\right)
$$

\vspace*{-20pt}

$$
\cO_2 = 
\left(
\barr{cccc}

\begin{tikzpicture}[scale=.5]
\fill(3,0) circle(.1);
\fill(2,1) circle(.1);
\fill(1,2) circle(.1);
\fill(0,3) circle(.1);
\fill(4,1) circle(.1);
\fill(2,3) circle(.1);
\draw(3,-.5) node{$1$};
\draw(1.5,1) node{$2$};
\draw(.5,2) node{$3$};
\draw(0,3.5) node{$4$};
\draw(2,3.5) node{$5$};
\draw(4,1.5) node{$5$};
\draw (4,1)--(3,0)--(0,3) (1,2)--(2,3);
\draw (5,0) node{,};
\end{tikzpicture}
&
\begin{tikzpicture}[scale=.5]
\fill(3,0) circle(.1);
\fill(2,1) circle(.1);
\fill(1,2) circle(.1);
\fill(0,3) circle(.1);
\fill(4,1) circle(.1);
\fill(2,3) circle(.1);
\draw(3,-.5) node{$1$};
\draw(1.5,1) node{$2$};
\draw(.5,2) node{$3$};
\draw(0,3.5) node{$5$};
\draw(2,3.5) node{$4$};
\draw(4,1.5) node{$4$};
\draw (4,1)--(3,0)--(0,3) (1,2)--(2,3);
\draw (5,0) node{,};
\end{tikzpicture}
&
\begin{tikzpicture}[scale=.5]
\fill(3,0) circle(.1);
\fill(2,1) circle(.1);
\fill(1,2) circle(.1);
\fill(0,3) circle(.1);
\fill(4,1) circle(.1);
\fill(2,3) circle(.1);
\draw(3,-.5) node{$1$};
\draw(1.5,1) node{$2$};
\draw(.5,2) node{$3$};
\draw(0,3.5) node{$4$};
\draw(2,3.5) node{$5$};
\draw(4,1.5) node{$3$};
\draw (4,1)--(3,0)--(0,3) (1,2)--(2,3);
\draw (5,0) node{,};
\end{tikzpicture}
&
\begin{tikzpicture}[scale=.5]
\fill(3,0) circle(.1);
\fill(2,1) circle(.1);
\fill(1,2) circle(.1);
\fill(0,3) circle(.1);
\fill(4,1) circle(.1);
\fill(2,3) circle(.1);
\draw(3,-.5) node{$1$};
\draw(1.5,1) node{$2$};
\draw(.5,2) node{$3$};
\draw(0,3.5) node{$5$};
\draw(2,3.5) node{$4$};
\draw(4,1.5) node{$2$};
\draw (4,1)--(3,0)--(0,3) (1,2)--(2,3);
\end{tikzpicture}

\earr
\right)
$$

\vspace*{-20pt}

$$
\cO_3 = 
\left(
\barr{cccc}

\begin{tikzpicture}[scale=.5]
\fill(3,0) circle(.1);
\fill(2,1) circle(.1);
\fill(1,2) circle(.1);
\fill(0,3) circle(.1);
\fill(4,1) circle(.1);
\fill(2,3) circle(.1);
\draw(3,-.5) node{$1$};
\draw(1.5,1) node{$2$};
\draw(.5,2) node{$3$};
\draw(0,3.5) node{$4$};
\draw(2,3.5) node{$4$};
\draw(4,1.5) node{$5$};
\draw (4,1)--(3,0)--(0,3) (1,2)--(2,3);
\draw (5,0) node{,};
\end{tikzpicture}
&
\begin{tikzpicture}[scale=.5]
\fill(3,0) circle(.1);
\fill(2,1) circle(.1);
\fill(1,2) circle(.1);
\fill(0,3) circle(.1);
\fill(4,1) circle(.1);
\fill(2,3) circle(.1);
\draw(3,-.5) node{$1$};
\draw(1.5,1) node{$2$};
\draw(.5,2) node{$3$};
\draw(0,3.5) node{$5$};
\draw(2,3.5) node{$5$};
\draw(4,1.5) node{$4$};
\draw (4,1)--(3,0)--(0,3) (1,2)--(2,3);
\draw (5,0) node{,};
\end{tikzpicture}
&
\begin{tikzpicture}[scale=.5]
\fill(3,0) circle(.1);
\fill(2,1) circle(.1);
\fill(1,2) circle(.1);
\fill(0,3) circle(.1);
\fill(4,1) circle(.1);
\fill(2,3) circle(.1);
\draw(3,-.5) node{$1$};
\draw(1.5,1) node{$2$};
\draw(.5,2) node{$4$};
\draw(0,3.5) node{$5$};
\draw(2,3.5) node{$5$};
\draw(4,1.5) node{$3$};
\draw (4,1)--(3,0)--(0,3) (1,2)--(2,3);
\draw (5,0) node{,};
\end{tikzpicture}
&
\begin{tikzpicture}[scale=.5]
\fill(3,0) circle(.1);
\fill(2,1) circle(.1);
\fill(1,2) circle(.1);
\fill(0,3) circle(.1);
\fill(4,1) circle(.1);
\fill(2,3) circle(.1);
\draw(3,-.5) node{$1$};
\draw(1.5,1) node{$3$};
\draw(.5,2) node{$4$};
\draw(0,3.5) node{$5$};
\draw(2,3.5) node{$5$};
\draw(4,1.5) node{$2$};
\draw (4,1)--(3,0)--(0,3) (1,2)--(2,3);
\end{tikzpicture}

\earr
\right)
$$
    \caption{The orbits for $T(2)$ when $m=5$}
    \label{bi}
\end{figure}
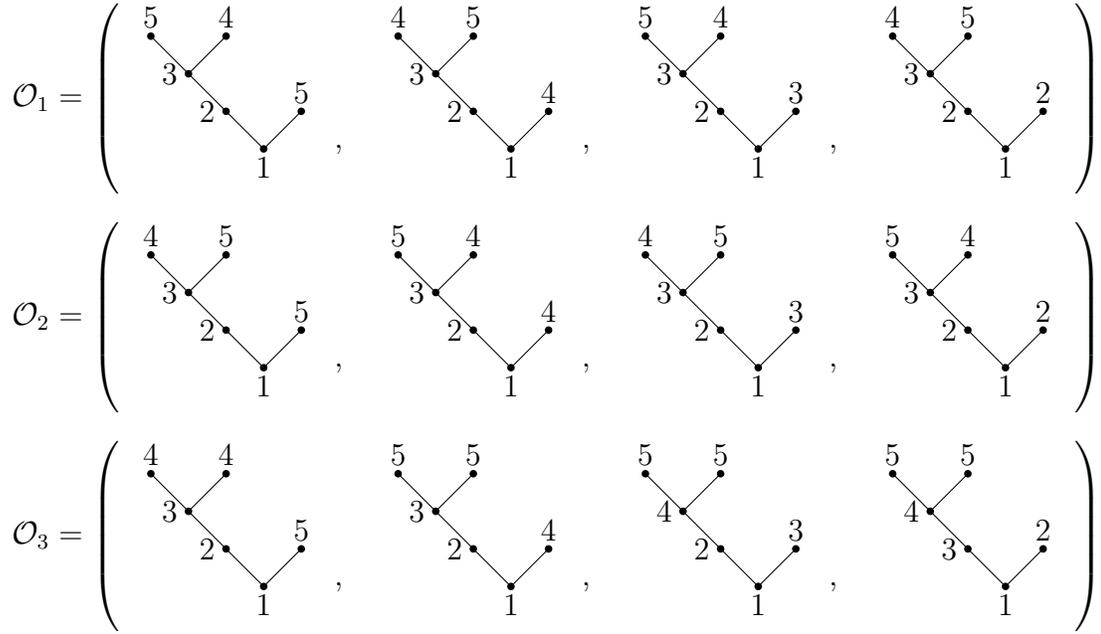

\begin{proof}
(a)  We will keep the leaf notation as in Figure~\ref{3Tree}.  
Since $m=n-2$ and there are $n-2$ elements on the path from $\zh$ to $u$, this path must be labeled $1,2,\ldots,m$.  The same is true of the path from $\zh$ to $v$.  So only the label of $w$ can vary when we apply $\pro_K$.  Since $w$ is a branch, the proof of Theorem~\ref{prin} (b) shows that its label cycles through the interval $[2,m]$.  These are clearly are all possible $m$-packed labelings, finishing the proof.

\medskip

(b)  Continuing to use the conventions in (a), we also let $x_1, x_2,\ldots x_c$ be the elements of the chain used to construct $T(c)$ from bottom to top, starting with the element covering $\zh$.  Again, we will explicitly construct the orbits in question.  An example of this construction will be found in Figure~\ref{bi}.

In case (i) we will show that there are three orbits
\begin{align}
\label{orb1}
\cO_1 &=(L_1,L_2,\ldots,L_{m-1}),\\
\label{orb2}
\cO_2  &=(M_1,M_2,\ldots,M_{m-1}),\\
\nonumber
\cO_3  &=(N_1,N_2,\ldots,N_{m-1}).
\end{align}

We will use $L$ for a generic labeling which may be in any of the orbits.  In all orbits $L(\zh)=1$ so only the other labels need to be specified.  In $\cO_1$ all the labelings have
\beq
\label{l(x_i)}
L(x_i) = i+1
\eeq
for $i\in [c]$. Furthermore, in $L_j$ for $j\in[m-1]$ we have
\beq
\label{l(w)}
L_{j}(w) = m - j + 1.
\eeq
As for the other two maximal elements, for $j\in[m-1]$ we have
\beq
\label{l(u)}
L_j(u) = \case{m}{if $j$ is odd,}{m-1}{if $j$ is even,}
\eeq
and
\beq
\label{l(v)}
L_j(v) = \case{m-1}{if $j$ is odd,}{m}{if $j$ is even.}
\eeq
It is easy to verify that the labelings just described form an orbit under $\pro_K$.

In $\cO_2$ we have that~\eqref{l(x_i)} and~\eqref{l(w)} are still satisfied.  But~\eqref{l(u)} and~\eqref{l(v)}are replaced by
$$
M_j(u) = \case{m-1}{if $j$ is odd,}{m}{if $j$ is even,}
$$
and
$$
M_j(v) = \case{m}{if $j$ is odd,}{m-1}{if $j$ is even.}
$$
Again, we leave the verification that the $M_j$ form an orbit to the reader.

Finally,  $N_1\in\cO_3$ satisfies~\eqref{l(x_i)}.   And~\eqref{l(w)} holds for all $N_j\in\cO_3$.   But
$$
N_1(u)=N_1(v) = m-1.
$$
All the other $N_j$ have
$$
N_j(u)=N_j(v) = m.
$$
Finally, for $N_j$ with $j\ge 2$ we label the chain $x_1,\ldots,x_c$ using the elements in the set difference $[2,m-1]-\{m-j+1\}$.  It is not difficult, if tedious, to use these explicit expressions to show that each $N_j$ will be an $m$-packed labeling and that $\cO_3$ is an orbit of $\pro_K$.  We also omit the details showing that these orbits contain all the $m$-packed labelings of $T(c)$ so that we have completely described the orbit structure.

In case (ii) we have the same labelings.  The only difference is that, because of the change in parity, orbits $\cO_1$ and $\cO_2$ merge into a single orbit
\beq
\label{merge}
\cO = (L_1,L_2,\ldots,L_{m-1},M_1,M_2,\ldots,M_{m-1})
\eeq

\medskip

(c)  Again, we will content ourselves with a description of the orbits.  Assume first that we are in case (i).  We want to describe orbits $\cO_1$ and $\cO_2$, keeping the notation of~\eqref{orb1} and~\eqref{orb2}.

In $\cO_1=(L_1,L_2,\ldots,L_{m-1})$,  leaf $w$ still has labels given by~\eqref{l(w)}.  For $j\le 2$ we have that~\eqref{l(x_i)} also holds.  Furthermore, we have
$$
L_1(u) = m-1 \text{ and } L_1(v)=m-2
$$
while
$$
L_2(u) =m-2 \text{ and }  L_2(v) = m.
$$
 Finally, for $j\ge3$, we have that $L_j$ satisfies~\eqref{l(u)} and~\eqref{l(v)}.  And the chain $x_1,\ldots,x_b$ takes the elements of
$[2,m-2]-\{n-i+1\}$ as its labels.

In $\cO_2=(M_1,M_2,\ldots,M_{m-1})$, the labels of $M_j$ are the same as those of $L_j$ except that the labels of $u$ and $v$ are reversed.  And when the parity of $c$ becomes odd in case (ii), the two orbits merge into one exactly as in~\eqref{merge}.
\end{proof}

\section{Questions and future directions}

We collect some  directions for future research on $m$-packed labelings of posets.

\subsection{More on combs and zippers}

\begin{table}
$$
\barr{c|r|l|r}
C_n & m &o_m(C_n)                                               & o(\pro_K)\\
\hline
\hline
C_3 & 4 & 6^1                                                   & 6\\
    & 5 & 8^1, 12^1                                             & 24\\
    & 6 & 15^1                                                  & 15\\
\hline
C_4 & 5 & 12^2                                                  & 12\\
    & 6 & 15^2, 40^1, 60^1                                      & 60\\
    & 7 & 30^3, 48^1, 72^1                                      &720\\
    & 8 & 35^3                                                  & 35\\
\hline
C_5 & 6 & 60^2                                                  & 60\\
    & 7 & 30^6, 60^6, 72^2, 120^2                               & 360\\
    & 8 & 35^6, 70^6, 140^2, 210^3, 336^1, 504^1                & 5040\\
    & 9 & 240^3, 280^3, 384^1, 576^1                            & 40320\\
    & 10& 315^3                                                 & 315\\
\hline
C_6 & 7 & 60^{12}                                               & 60\\
    & 8 & 70^{12}, 210^6, 420^6, 504^2, 840^2                   & 2520\\
    & 9 & 240^6, 280^{26}, 336^8, 480^6, 504^8, 576^2, 840^6, 960^2 & 20160\\
    & 10& 315^{26}, 378^8, 567^8, 576^9, 945^6, 720^9, 1152^3, 2520^3 & 362880\\
    & 11& 630^{15}, 640^9, 800^9, 1280^3, 2800^3                & 403200\\
    & 12& 693^{15}                                              & 693
\earr    
$$
    \caption{Comb orbit sizes and orders for $\pro_K$}
    \label{C_nTab}
\end{table}

In Theorem~\ref{C:2n} and Proposition~\ref{C:n+1Z:n+2} (a) we were able to characterize the orbit sizes of $\pro_K$ acting on the comb $\cL_m(C_n)$ when $m$ has one of the two extreme values $\#C_n=2n$ and $h(C_n)+1=n+1$.
And in Proposition~\ref{C:n+1Z:n+2} (b) we did the same for the zipper $Z_n$ when $m=h(Z_n)+1=n+2$.  Computer data suggests that other values of $m$ may yield interesting results.  In Table~\ref{C_nTab} we have listed the orbit sizes and order for $\pro_K$ acting on $C_n$ for $3\le n \le 6$.  We use the notation from  \eqref{omP} and indicate multiplicities of orbit sizes as $k^l$ if there are $l$ orbits of size $k$. Table~\ref{Z_nTab} contains similar information about zippers.

\subsection {The cyclic sieving phenomenon}

Cyclic sieving involves three objects.  Let $S$ be a set, $C$ be a cyclic group acting on $S$, and $f(q)$ be a polynomial in $q$ with coefficients in $\bbN$.  We say that the triple $(S,C,f(q))$ exhibits the {\em cyclic sieving phenomenon} (CSP) if, for all $g\in C$ we have
$$
\#S^g = f(\om)
$$
where $S^g$ is the fixed point set of $g$ and $\om$ is a root of unity chosen to have the same order as $g$.  The CSP was first defined by Reiner, Stanton, and White~\cite{RSW:csp} and has since found wide application.
See the  article of Sagan~\cite{sag:CSP} for a survey.
In Pechenik's original article~\cite{pec:csi}, he used $K$-promotion on $m$-packed labelings of a Young diagram of rectangular shape $[2]\times[n]$ to prove a CSP for these tableaux.  In fact, instances of the CSP are often connected with a form of promotion; the survey articles~\cite{rob:dac,str:dac} contain more information.

\begin{table}
$$
\barr{c|r|l|r}
Z_n & m &o_m(Z_n)                                               & o(\pro_K)\\
\hline
\hline
Z_1 & 3 & 1^1                                                   & 1\\
    & 4 & 2^2, 3^2, 6^2                                         & 6\\
    & 5 & 2^1, 4^1, 8^{12}                                      & 8\\
    & 6 & 10^{16}                                               & 10\\
    & 7 & 4^2, 12^6                                             & 12\\
\hline
Z_2 & 4 & 2^2                                                   & 2\\
    & 5 & 3^3, 6^{12}, 8^6, 24^6                                & 24\\
    & 6 & 8^8, 10^{12}, 15^{108}, 30^{36}, 40^8, 120^6          & 120\\
    & 7 & 4^2, 9^{54}, 12^6, 18^{378}, 36^{60}, 48^{80}. 144^{60} & 144\\
    & 8 & 21^{540}, 42^{30}, 56^{300}, 168^{180}                & 168\\
    & 9 & 6^3, 12^{39}, 24^{216}, 64^{588}, 192^{210}           & 192\\
    & 10& 72^{560}, 2216^{84}                                   & 216\\
    & 11& 16^8, 80^{200}                                        & 80
\earr    
$$
    \caption{Zipper orbit sizes and orders for $\pro_K$}
    \label{Z_nTab}
\end{table}

One could ask whether $\pro_K$ acting on $m$-packed labelings of rooted trees $T$ has an associated CSP.  For this, one would need a polynomial and, in the case $m=\#T$, a natural candidate is the $q$-analogue of formula~\eqref{hook}.    
To turn this quantity  into  a polynomial in $q$ we use the standard $q$-analogue of $n\in\bbN$ which is
$$
[n]_q = 1 + q  + q^2 +\cdots + q^{n-1}.
$$
Now define
\beq
\label{f^T(q)}
f^T(q) = \frac{[n]_q!}{\prod_{x\in T} [h_x]_q}
\eeq
where $[n]_q! = [1]_q [2]_q\cdots[n]_q$.
This turns out to be a polynomial in $q$ and is, up to a power of $q$, the generating function for various statistics on trees as shown by Bj\"orner and Wachs~\cite{BW:qlf}.

Unfortunately, \eqref{f^T(q)} can not be the polynomial for a CSP on natural labelings of an arbitrary rooted tree.  As an example, take the comb $C_3$.  By~\eqref{hook}, there are $15$ standard labelings of $C_3$.  And it is easy to show that there is only one orbit.  It follows that a single application of $\pro_K$ has no fixed points.  On the other hand, $f^{C_3}(q) = [5]_q [3]_q$ by~\eqref{f^T(q)}.  Substituting a primitive $15$th root of unity into this polynomial does not give zero.

This raises a number of questions.  Can one characterize the rooted trees for which~\eqref{f^T(q)} does give a CSP?  Is there another $q$-analogue of~\eqref{hook} which will yield a CSP for all standard labelings of trees?  Is there a CSP for $m$-packed labelings of trees when $m<\#T$?

\subsection{Other posets}

In the present work, we have considered the orbit structure of $\pro_K$ acting on $m$-packed labelings of rooted trees.  But other posets may also yield interesting results.  As mentioned in the introduction, the generalization of $\pro_K$ to increasing labelings has received some attention.  And the orbits on $m$-packed labelings are a subset of those on increasing ones.  So, restricting theorems for increasing labelings of $P$ may yield information about the action on $\cL_m(P)$.  For example, one could restrict the action of $\pro_K$ on increasing labelings of a product of two chains (equivalently, increasing Young tableaux of rectangular shape) to the $m$-packed case.

Often posets where rowmotion has a nice orbit structure also behave well with respect to some form of promotion.  In fact, it was the results on rowmotion on rooted trees in~\cite{DKLLSS:rrt} which partially motivated our work.  Another family of posets with interesting rowmotion  structure is fences.  
A {\em fence}, $F$, is a poset on the elements $x_1, x_2, \ldots, x_n$ whose cover relations are
$$
x_1\lt x_2 \lt \ldots \lt x_k \gt x_{k+1} \gt \ldots\gt  x_l \lt x_{l+1} \ldots
$$
for certain $k,l,\ldots\in[n]$.  Fences appear in the theory of cluster algebras as well as the work of Morier-Genoud and Ovsienko~\cite{MGO:qrq} on $q$-analogues of rational numbers.  In~\cite{EPRS:rf}, Elizalde et al.\ have derived a number of  results about rowmotion for fences.  It would be interesting to study the action of $\pro_K$ in this context.

\subsection{Homomesy and homometry}

We keep the notation in Section~\ref{es} concerning homomesy.  There is a weaker condition that is also displayed by some group actions.  We say that the statistic $\st:S\ra\bbN$ is {\em homometric} if, for any two orbits $\cO_1$ and $\cO_2$ for the action of $G$,
$$
\#\cO_1 = \#\cO_2 \implies \st \cO_1 = \st \cO_2.
$$
Note that homomesy implies homometry.  Indeed, if equation~\eqref{HomEq} is satisfied and $\#\cO_1=\#\cO_2$ then
$$
\st \cO_1 = c\cdot\#\cO_1 =  c\cdot\#\cO_2   = \st \cO_2.
$$
But there are statistics which are homometric but not homomesic.  
Homometry was first defined and studied in~\cite{EPRS:rf}.

Just as we used Theorem~\ref{ProRhoEq} to lift a homomesic result about rowmotion from~\cite{DKLLSS:rrt} to $\pro_K$ in Corollary~\ref{bSta}, one could do the same thing for homometries.  It would also be interesting to go the other way and find natural homomesies and homometries in the context of $K$-promotion and then see what they implied about rowmotion.  To apply Theorem~\ref{ProRhoEq}, the poset $P$ must have all chains of the same length and that is true for a number of interesting posets such as the product of chains or certain fences.  To use this theorem, one also needs to have $m=h(P)+2$, but there could well be interesting homomesies and homometries for other values of $m$.

\medskip
\textit{Acknowledgments.}
The authors would like to thank Rebecca Patrias for inspiring the topic of this paper.  We also would like to thank Jessica Striker and Oliver Pechenik for reading a preliminary version of this paper and giving us many comments which significantly improved it.  Finally, we are grateful to the referee who made suggestions for improvement of the exposition and caught an error in one of the theorems.

\nocite{*}
\bibliographystyle{alpha}

\newcommand{\etalchar}[1]{$^{#1}$}

\end{document}